\documentclass[12pt]{amsart}
\usepackage{amssymb,amscd}
\usepackage{verbatim}

\usepackage{amsmath,amssymb,graphicx,mathrsfs}   
\usepackage[colorlinks=true,allcolors = blue]{hyperref} 

\let\frak\mathfrak

\def\>{\relax\ifmmode\mskip.666667\thinmuskip\relax\else\kern.111111em\fi}
\def\<{\relax\ifmmode\mskip-.333333\thinmuskip\relax\else\kern-.0555556em\fi}
\def\vsk#1>{\vskip#1\baselineskip}
\def\vv#1>{\vadjust{\vsk#1>}\ignorespaces}
\def\vvn#1>{\vadjust{\nobreak\vsk#1>\nobreak}\ignorespaces}

  \let\ssize\scriptstyle
\let\sssize\scriptscriptstyle

\let\Medskip\medskip
\def\medskip{\par\Medskip}
\let\Bigskip\bigskip
\def\bigskip{\par\Bigskip}

\let\Maketitle\maketitle
\def\maketitle{\Maketitle\thispagestyle{empty}\let\maketitle\empty}

\newtheorem{thm}{Theorem}[section]
\newtheorem{cor}[thm]{Corollary}
\newtheorem{lem}[thm]{Lemma}

\newtheorem{ex}[thm]{Example}

\theoremstyle{definition}                                  
\numberwithin{equation}{section}

\theoremstyle{definition}
\newtheorem*{rem}{Remark}

\let\mc\mathcal
\let\nc\newcommand

\let\al\alpha
\let\bt\beta
\let\dl\delta

\let\la\lambda

\let\phi\varphi

\let\der\partial

\let\geq\geqslant

\let\leq\leqslant

\let\on\operatorname
\let\bi\bibitem
\let\bs\boldsymbol

\def\C{{\mathbb C}}
\def\Z{{\mathbb Z}}

\def\F{{\mathbb F}}   

\def\+#1{^{\{#1\}}}

\def\beq{\begin{equation}}
\def\eeq{\end{equation}}
\def\be{\begin{equation*}}
\def\ee{\end{equation*}}

\nc{\bea}{\begin{eqnarray*}}
\nc{\eea}{\end{eqnarray*}}
\nc{\bean}{\begin{eqnarray}}
\nc{\eean}{\end{eqnarray}}

\let\ga\gamma
\let\Ga\Gamma

\nc{\Il}{{\mc I_{\bs\la}}}
\nc{\bla}{{\bs\la}}
\nc{\Fla}{\F_\bla}
\nc{\tfl}{{T^*\Fla}}
\nc{\GL}{{GL_n(\C)}}
\nc{\GLC}{{GL_n(\C)\times\C^*}}

\let\sd s 

\def\ddk_#1{\kk_{#1}\<\>\frac\der{\der\<\>\kk_{#1}}}

\def\bul{\mathbin{\raise.2ex\hbox{$\sssize\bullet$}}}
\def\intt{\mathchoice
{\mathop{\raise.2ex\rlap{$\,\,\ssize\backslash$}{\intop}}\nolimits}
{\mathop{\raise.3ex\rlap{$\,\sssize\backslash$}{\intop}}\nolimits}
{\mathop{\raise.1ex\rlap{$\sssize\>\backslash$}{\intop}}\nolimits}
{\mathop{\rlap{$\sssize\<\>\backslash$}{\intop}}\nolimits}}

\let\kk q 
\let\cc c

\let\Ko K

\def\GZ/{Gelfand-Zetlin}
\def\KZ/{{\slshape KZ\/}}
\def\qKZ/{{\slshape qKZ\/}}
\def\XXX/{{\slshape XXX\/}}

\nc{\A}{{\mc C}}

\nc{\hsl}{\widehat{{\frak{sl}_2}}}

\nc{\BC}{{ \mathbb C}}
\nc{\lra}{\longrightarrow}
\nc{\CO}{{\mathcal{O}}}
\nc{\BZ}{{ \mathbb Z}}
\nc{\hfn}{\hat{\frak{n}}}

\begin{document}

\hrule width0pt
\vsk->
\title[Notes on $2D$ $\F_p$-Selberg integrals]
{Notes on $2D$ $\F_p$-Selberg integrals}

\author[Alexander Varchenko]
{Alexander Varchenko}

\maketitle

\begin{center}

{\it Department of Mathematics, University
of North Carolina at Chapel Hill\\ Chapel Hill, NC 27599-3250, USA\/}

\end{center}

{\let\thefootnote\relax
\footnotetext{\vsk-.8>\noindent
{\sl E-mail}:\enspace anv@email.unc.edu
}}

\begin{abstract}

We prove a two-dimensional $\F_p$-Selberg integral formula, in which
the two-dimensional $\F_p$-Selberg integral $\bar S(a,b,c;l_1,l_2)$
depends on positive integer parameters $a,b,c$,  $l_1,l_2$ and 
is an element of the finite field $\F_p$ with odd prime number $p$ of  elements.
The formula is motivated by the analogy between  multidimensional hypergeometric solutions of the KZ equations and
 polynomial solutions of the same equations reduced modulo $p$.

\end{abstract}

{\small\tableofcontents\par}

\section{Introduction}
In 1944 Atle Selberg proved the following integral formula:
\bean
\label{clS}
&&
\int_0^1\dots\int_0^1 \prod_{1\leq i<j\leq n} (x_i-x_j)^{2\ga} \prod_{i=1}^nx_i^{\al-1} (1-x_i)^{\beta-1}\ dx_1\dots dx_n
\\
&&
\notag
\phantom{aaaaaa}
=\
\prod_{j=1}^n \frac{\Ga(1+j\ga)}{\Ga(1+\ga)}\,
\frac{\Ga(\al+(j-1)\ga)\,\Ga(\beta+(j-1)\ga)}
{\Ga(\al+\beta + (n+j-2)\ga)}\,,
\eean
see  \cite{Se, AAR}.
Hundreds of papers are devoted to the generalizations of the Selberg
integral formula and its applications, see for example \cite{AAR, FW} and references 
therein.  There are $q$-analysis versions of the formula,
the generalizations associated with Lie algebras, elliptic versions, finite field versions,
see some references  in \cite{AAR, FW}.
In the finite field versions,
 one considers additive and multiplicative characters of a finite field, which 
map the  field to the field of
complex numbers, and forms an analog of equation \eqref{clS}, in which  both sides are complex numbers.
The simplest of such formulas is the classical relation between Jacobi and Gauss 
sums, see \cite{AAR}.

\vsk.2>
In \cite{RV1}, another  version of the Selberg integral formula
was presented, in which the
$\F_p$-Selberg integral is an element of the  finite field $\F_p$ with an odd prime number $p$ 
 of elements, see also \cite{RV2}.  Given non-negative integers  $a,b,c$, consider the master polynomial
$\Phi_n \in \F_p[x_1,\dots,x_n]$,
\bea
\Phi_n = \prod_{1\leq i<j\leq n} (x_i-x_j)^{2c} \prod_{i=1}^nx_i^a (1-x_i)^b.
\eea
Denote by $\bar S(a,b,c)$ the coefficient of the monomial $x_1^{p-1} \dots x_n^{p-1}$ in $\Phi_n$
and call it the $\F_p$-Selberg integral.

\begin{thm} [{\cite[Theorem 4.1]{RV1}}]
\label{thm nd}

Assume that $a,b,c$ are non-negative integers such that
\bean
\label{abc n eq}
 p-1\leq a+b+(n-1)c,
\qquad a+b+(2n-2)c < 2p-1\ .
\eean
Then we have a formula in $\F_p$:
\bean
\label{main n}
\bar S(a,b,c) =
(-1)^n\,\prod_{j=1}^n \frac{(jc)!}{c!}\,
\frac{(a+(j-1)c)!\,(b+(j-1)c)!}
{(a+b + (n+j-2)c+1-p)!}\,.
\eean
\end{thm}

The master polynomial $\Phi_n$ is an analog of the integrand in \eqref{clS}.
The operation of choosing the coefficient 
of $x_1^{p-1} \dots x_n^{p-1}$ in $\Phi_n$ is an analog of the integration 
of $\Psi_n$ over a cycle
 due to the following Stokes-like observation. For any $\Psi \in \F_p[x_1,\dots,x_n]$ and any positive
integers $l_1,\dots,l_n$, the coefficient of
$x_1^{l_1p-1} \dots x_n^{l_np-1}$  in  any first partial derivative 
$\frac {\der \Psi}{\der x_i}$  equals zero.

\vsk.2>

In this paper we consider the case $n=2$. For positive integers 
$a,b,c$,\ $l_1,l_2$ we denote by  $\bar S(a,b,c;l_1,l_2)$ the coefficient of 
$x_1^{l_1p-1}x_2^{l_2p-1}$ in $\Phi_2$ and call it a two-dimensional $\F_p$-Selberg integral.
Clearly, $\bar S(a,b,c;l_1,l_2)=\bar S(a,b,c;l_2,l_1)$,

\vsk.2>

We assume that $0<a,b,c<p$ and evaluate $\bar S(a,b,c;l_1,l_2)$ in all non-zero cases.  This is the main result of this paper. 
It is interesting  that  in all cases, the $\F_p$-Selberg integral 
$\bar S(a,b,c;l_1,l_2)$  is given by a formula analogous  to formula \eqref{main n} with some shifts
by $p$ in factorials.       

\vsk.2>

Here, in the introduction  we formulate a theorem that
lists all the integers $0<a,b,c<p$ such  that there are more than one pair $l_1\leq l_2$ with non-zero
$\bar S(a,b,c;l_1,l_2)$, see Theorem \ref{thm rels}.

\begin{thm}
\label{thm intro}

If   there are more than one pair $l_1\leq l_2$ such that
$\bar S(a,b,c;l_1,l_2)$ is non-zero, then all such $(a,b,c;l_1,l_2)$ are listed below.
\begin{enumerate}
\item[(i)]

If $2c<p, \ a+c\leq p-1, \ b+c\geq p,  \  a+b+2c\geq 2p-1$,
then
$\bar S(a,b,c;1,1)$,  $\bar S(a,b,c;1,2)$, $\bar S(a,b,c;2,1)$ are non-zero and
\bean
\label{irel 11-12}
&&
\phantom{aaaa}
-\frac 12\,\bar S(a,b,c;1,1) = \bar S(a,b,c;1,2) = \bar S(a,b,c;2,1),
\\
\label{iexc1}
&&
\bar S(a,b,c;1,1) = \frac{(2c)!}{c!}\,\frac{a!\,(a+c)!\,b!\,(b+c-p)!}{(a+b+c-p+1)!\,(a+b+2c-2p+1)!}\,.
\eean

\item[(ii)]

If $2c<p, \  a+b+c \geq  2p-1$,
 then
$\bar S(a,b,c;2,2)$,  $\bar S(a,b,c;1,2)$, $\bar S(a,b,c;2,1)$ are non-zero and
\bean
\label{irel 22-12}
&&
\phantom{aaaa}
-\frac 12\,\bar S(a,b,c;2,2) = \bar S(a,b,c;1,2) = \bar S(a,b,c;2,1).
\\
\label{iexc2}
&&
\bar S(a,b,c;2,2) =
 - \frac{(2c)!}{c!}\,\frac{a!\,(a+c-p)!\,b!\,(b+c-p)!}{(a+b+c-2p+1)!\,(a+b+2c-2p+1)!}\,.
\eean

\item[(iii)]

If  $2c>p, \quad   a+b+2c\geq  3p-1$, 
then
$\bar S(a,b,c;2,2)$, $\bar S(a,b,c;1,3)$, $\bar S(a,b,c;3,1)$
are non-zero and
\bean
\label{irel 22-13}
&&
\phantom{aaaa}
-\frac 12\,\bar S(a,b,c;2,2) =\bar S(a,b,c;1,3) = \bar S(a,b,c;3,1),
\\
\label{iexc3}
&&
\bar S(a,b,c;2,2) =
 - \frac{(2c-p)!}{c!}\,\frac{a!\,(a+c-p)!\,b!\,(b+c-p)!}{(a+b+c-2p+1)!\,(a+b+2c-3p+1)!}\,.
\eean

\end{enumerate}

\end{thm}

If $(a,b,c)$  does not satisfy the system of inequalities,
$2c<p, \ a+c\leq p-1, \ b+c\geq p,  \  a+b+2c\geq 2p-1$,
and does not satisfy the  system of  inequalities 
$2c<p, \ a+b+c \geq  2p-1$,  
and does not satisfy the  system of  inequalities,
$2c>p, \   a+b+2c\geq  3p-1$,
then there exists at most one pair
$(l_1,l_2)$ with $l_1\leq l_2$ such that $\bar S(a,b,c;l_1,l_2)\ne 0$.

\vsk.2>

It would be interesting to prove relations \eqref{irel 11-12} and \eqref{irel 22-12} without evaluating their terms, see
Section \ref{sec 4.8} where that was done for the relations \eqref{irel 22-13}.

It would also be  interesting to evaluate all $n$-dimensional $\F_p$-Selberg integrals
\\
$\bar S(a,b,c;l_1,\dots, l_n)$ for arbitrary $n$ and find all relations between them.

\vsk.2>

This paper is a part of the study of solutions of the KZ equations in finite characteristic, see for example 
\cite{SV2, EV, VV}.
Formulas in this paper are 
 motivated by the analogy between  multidimensional hypergeometric solutions of the KZ equations and
 polynomial solutions of the same equations reduced modulo $p$, cf. \cite{SV1, SV2, VV}.

\smallskip

In Section \ref{sec 2} we collect useful facts. In Section \ref{sec 3} 
we classify all the cases in which $\bar S(a,b,c;l_1,l_2)$ is non-zero and evaluate it.

\smallskip
The author thanks A.\,Slinkin and V.Vologodsky for useful discussions.

\section{Preliminary remarks}
\label{sec 2}

\subsection{Lucas' Theorem}

\begin{thm} [\cite{L}] \label{thm L}
For non-negative integers $m$ and $n$ and a prime $p$,
 the following congruence relation holds:
\bean
\label{Lucas}
\binom{n}{m}\equiv \prod _{i=0}^a \binom{n_i}{m_i}\quad (\on{mod}\ p),
\eean
where $m=m_{b}p^{b}+m_{b-1}p^{b-1}+\cdots +m_{1}p+m_{0}$ 
and  $n=n_{b}p^{b}+n_{b-1}p^{b-1}+\cdots +n_{1}p+n_{0}$
are the base $p$ expansions of $m$ and $n$ respectively. This uses the convention that 
$\binom{n}{m}=0$ if $n<m$.
\qed
\end{thm}

\subsection{Cancellation of factorials}

\begin{lem}
\label{lem ca}

If $a,b$ are non-negative integers and  $a+b=p-1$, then in $\F_p$ we have
\bean
\label{ca}
a!\,b!\,=\,(-1)^{a+1}\,.
\eean

\end{lem}

\begin{proof} We have $a!=(-1)^a(p-1)\dots(p-a)$ and $p-a= b+1$. Hence
$a!\,b! = (-1)^a(p-1)!= (-1)^{a+1}$ by Wilson's Theorem.
\end{proof}

\begin{lem}
\label{lem ga}
Let $a,b$ be positive integers such that $a<p$, $b<p$, $p\leq a+b$. Then we have an identity in $\F_p$,
\bean
\label{p-ga}
  b\,\binom{b-1}{a+b-p} 
  =
b\,\binom{b-1}{p-a-1}
=
 (-1)^{a+1}\,\frac{a! \,b!}
{(a+b-p)!}\,.
\eean
\qed

\end{lem}

\subsection{$\F_p$-Integrals}
\label{sec 3}

Let $M$ be an  $\F_p$-module. Let $P(x_1,\dots,x_k)$ be a polynomial
with coefficients in $M$,
\bean
\label{St}
P(x_1,\dots,x_k) = \sum_{d}\, c_d \,x_1^{d_1}\dots x_k^{d_k}.
\eean
Let $l=(l_1,\dots,l_k)\in \Z_{>0}^k$. We call the coefficient
$c_{l_1p-1,\dots,l_kp-1}$  the {\it $\F_p$-integral} of the polynomial $P$ over the cycle $[l_1,\dots,l_k]_p$
and denoted it by $\int_{[l_1,\dots,l_k]_p} P(x_1,\dots,x_k)\,dx_1\dots dx_k$.

\begin{lem}
\label{lem St}
For any $i=1,\dots,k$, we have
\bea
\int_{[l_1,\dots,l_k]_p} \frac{\der P}{\der x_i}(x_1,\dots,x_k) \,dx_1\dots dx_k= 0\,.
\eea
\qed
\end{lem}

\subsection{$\F_p$-Beta integral}
For non-negative integers $a, b$ the classical beta integral formula says
\bean
\label{bk}
\int_{0}^{1} x^{a}(1-x)^{b} dx =\frac{a!\,b!}{(a+b+1)!}\,.
\eean

\begin{lem} [\cite{V}]
\label{lem bfun}
Let
$0\leq a, b <p$, $p-1\leq a+b$. 
Then in $\F_p$ we have
\bean
\label{bf1}
\int_{[1]_p} x^a(1-x)^b dx\,
=  \, - \,\frac{a!\,b!}{(a+b-p+1)!}\,.
\eean
If $a+b<p-1$, then 
\bean
\label{bf2}
\int_{[1]_p} x^a(1-x)^b dx\,
=  0\,.
\eean

\end{lem}

\begin{proof}
We have
$
x^a(1-x)^b = \sum_{k=0}^b(-1)^k \binom{b}{k} x^{k}\,,
$
and need  $a+k=p-1$. Hence $k=p-1-a$ and 
\bea
\int_{[1]_p} x^a(1-x)^bdx \,
= (-1)^{p-1-a}\binom{b}{p-1-a}.
\eea
Now Lemma \ref{lem ga} implies \eqref{bf1}. Formula \eqref{bf2} is clear.
\end{proof}

\subsection{Morris' identity}

Suppose that $\al,\beta,\ga$ are non-negative integers. Then
\bean
\label{Mid}
&&
\on{CT}\,\prod_{i=1}^n (1-x_i)^\al (1-1/x_i)^\beta \prod_{1\leq j\ne k\leq n} (1-x_j/x_k)^\ga
\\
\notag
&&
\phantom{a}
=\, \ \prod_{j=1}^n \frac{(j\ga)!}{\ga!}\,
\frac{(\al+\beta +(j-1)\ga)!}
{(\al+(j-1)\ga)!\,(\beta +(j-1)\ga)!}\,,
\eean
 where CT denotes the constant term.
 Morris identity was deduced in \cite{Mo} from the integral formula 
for the classical Selberg integral, see \cite[Section 8.8]{AAR}.

The left-hand side of  \eqref{Mid} can be written as 
\bean
\label{Mim}
\on{CT}\,(-1)^{\binom{n}{2}\ga + n\beta}
\prod_{1\leq i<j\leq n}(x_i-x_j)^{2\ga}
\prod_{i=1}^n x_i^{-\beta-(n-1)\ga}(1-x_i)^{\al+\beta}\,.
\eean

\section{$2D$ $\F_p$-Selberg integrals}
\label{sec 3}

\subsection{Definition}

For integers $a,b,c$, 
\bean
\label{less p}
0 \ <\  a,b,c  \ <\ p,
\eean
introduce the master polynomial
\bea
\Phi (x_1,x_2; a,b,c)
&=&
(x_1-x_2)^{2c} \prod_{i=1}^2x_i^a (1-x_i)^b 
\eea
as a polynomial in $\Z[x_1,x_2]$. 
For positive integers $l_1,l_2$, denote by $S(a,b,c;l_1,l_2) $ the coefficient of $x_1^{l_1p-1}x_2^{l_2p-1}$ 
in $\Phi (x_1,x_2; a,b,c)$. Denote by $\bar S(a,b,c;l_1,l_2) $ the projection of
\\
 $S(a,b,c;l_1,l_2) $ to $\F_p$.

We have in $\F_p$,
\bea
\bar S(a,b,c;l_1,l_2) 
&=& 
\int_{[l_1,l_2]_p} \Phi (x_1,x_2;a,b,c)dx_1dx_2 \,.
\eea
The element $\bar S(a,b,c;l_1,l_2) $ is
called 
 a two-dimensional {\it $\F_p$-Selberg integral.}

\vsk.2>

We have $S(a,b,c;l_1,l_2) = S(a,b,c;l_2,l_1)$ 
since $\Phi (x_1,x_2; a,b,c)= \Phi (x_2,x_1; a,b,c)$.

\subsection{Recursion}  
Denote
\bea
\bar S_1(a,b,c;l_1,l_2) 
&=&
 \int_{[l_1,l_2]_p} (x_1+x_2)\Phi(x;a,b,c) dx_1dx_2\,,
\\
\bar S_2(a,b,c;1,2) 
&=&
 \int_{[l_1,l_2]_p} ((1-x_1)+(1-x_2))\Phi(x;a,b,c) dx_1dx_2\,.
\eea
These are elements of $\F_p$.
 
 \begin{thm}
 \label{thm Ao}
 We have 
 \bean
\label{Ao1}
(a+1) \,\bar S_1(a,b,c;l_1,l_2) 
&=&
2(a+b+c+2)\,
\bar S(a+1,b,c;l_1,l_2),
\\
\label{Ao2}
2(a+c+1) \,\bar S(a,b,c;l_1,l_2) 
&=&
(a+b+2c+2)\,
\bar S_1(a,b,c;l_1,l_2),
\eean
 \bean
\label{Ao3}
(b+1) \,\bar S_2(a,b,c;l_1,l_2) 
&=&
2(a+b+c+2)\,
\bar S(a,b+1,c;l_1,l_2),
\\
\label{Ao4}
2(b+c+1) \,\bar S(a,b,c;l_1,l_2) 
&=&
(a+b+2c+2)\,
\bar S_2(a,b,c;l_1,l_2).
\eean

  \end{thm}

\begin{proof}

The proof is similar to the proofs in \cite[Section 4.4]{RV1}. 
Adding the equations
 \bea
0
&=&
\int_{[l_1,l_2]_p}\frac{\der}{\der x_1}\big[(1-x_1)x_1x_2 \Phi\big] dx_1dx_2
\\
&=&
\int_{[l_1,l_2]_p} \Phi \Big[-(b+1) x_1x_2 + (a+1) (1-x_1)x_2 + 2c\frac{x_1x_2(1-x_1)}{x_1-x_2}\Big]dx_1dx_2\,,
\\
0
&=&
\int_{[l_1,l_2]_p}\frac{\der}{\der x_2}\big[(1-x_2)x_1x_2 \Phi\big] dx_1dx_2
\\
&=&
\int_{[l_1,l_2]_p} \Phi \Big[-(b+1) x_1x_2 + (a+1) (1-x_2)x_1 + 2c\frac{x_1x_2(1-x_2)}{x_2-x_1}\Big]dx_1dx_2\,
\eea
 we obtain \eqref{Ao1}.   Adding the equations
\bea
0
&=&
\int_{[l_1,l_2]_p}\frac{\der}{\der x_1}\big[(1-x_1)x_1 \Phi\big] dx_1dx_2
\\
&=&
\int_{[l_1,l_2]_p} \Phi \Big[-(b+1) x_1 + (a+1) (1-x_1) + 2c\frac{x_1(1-x_1)}{x_1-x_2}\Big]dx_1dx_2\,,
\\
0
&=&
\int_{[l_1,l_2]_p}\frac{\der}{\der x_2}\big[(1-x_2)x_2 \Phi\big] dx_1dx_2
\\
&=&
\int_{[l_1,l_2]_p} \Phi \Big[-(b+1) x_2 + (a+1) (1-x_2) + 2c\frac{x_2(1-x_2)}{x_2-x_1}\Big]dx_1dx_2\,
\eea
we obtain \eqref{Ao2}. Equations \eqref{Ao3} and \eqref{Ao4} are proved similarly.
\end{proof}

\begin{cor} We have
\bean
\label{Ar1}
\bar S(a,b,c;l_1,l_2) &=& \bar S(a-1,b,c;l_1,l_2)
\,\frac{a\,(a+c)}{(a+b+c+1)\,(a+b+2c+1)},
\eean
if the denominator is non-zero, and
\bean
\label{Ar2}
\bar S(a,b,c;l_1,l_2) &=& \bar S(a,b-1,c;l_1,l_2)
\,\frac{b\,(b+c)}{(a+b+c+1)\,(a+b+2c+1)},
\eean
if the denominator is non-zero.
\qed
\end{cor}

\subsection{$p$-cycle $[1,1]_p$}

In this section we evaluate $\bar S(a,b,c;1,1)$.

\begin{lem}
\label{lem i11}
We have $ S(a,b,c;1,1)=0$ if at least one of the following inequalities holds:
\bean
\label{11 ine}
p\leq a+c,  \qquad a+b+c \leq p-2.
\eean

\end{lem}

\begin{proof}

If $p\leq a+c$, then for every monomial $x_1^{d_1}x_2^{d_2}$ of $(x_1-x_2)^{2c}x_1^{a}x_2^{a}$ we have
$\max(d_1,d_2)\geq p$, and the monomial $x_1^{p-1}x_2^{p-1}$ does not enter the master polynomial.

Similarly, if $a+b+c\leq  p-2$, then the monomial $x_1^{p-1}x_2^{p-1}$ does not enter the master polynomial. 
\end{proof}

\begin{thm}
\label{thm 2d1}

Assume that $a,b,c$ satisfy \eqref{less p} and the  system of inequalities
\bean
\label{11 inee}
a+c\leq p-1,  \qquad a+b+c \geq p-1.
\eean
 Then the following statements hold true.
\begin{enumerate}
\item[(i)] If  $b+c \leq p-1$, then 
\bean
\label{11}
\bar S(a,b,c;1,1) = \frac{(2c)!}{c!}\,\frac{a!\,(a+c)!\,b!\,(b+c)!}{(a+b+c-p+1)!\,(a+b+2c-p+1)!}\,.
\eean
This expression is non-zero if and only if $2c<p$.

\item[(ii)] If $b+c \geq p$ and $a+b+2c\geq 2p-1$, then 
\bean
\label{12}
\bar S(a,b,c;1,1) = \frac{(2c)!}{c!}\,\frac{a!\,(a+c)!\,b!\,(b+c-p)!}{(a+b+c-p+1)!\,(a+b+2c-2p+1)!}\,.
\eean
This expression is non-zero if and only if $2c<p$.

\item[(iii)]   If $b+c \geq p$ and $a+b+2c\leq  2p-2$, then $\bar S(a,b,c;1,1)=0$.

\end{enumerate}

\end{thm}

\begin{rem}

Part (i) of Theorem \ref{thm 2d1} is a particular case of \cite[Theorem 4.1]{RV1} for $n=2$.

\end{rem}

\begin{proof} 

Proof of part (i).
We have 
\bean
\label{SnM}
S(a,b,c;1,1) \,=\,
\on{CT}\,
(x_1-x_2)^{2c}
\prod_{i=1}^2 x_i^{a+1-p}(1-x_i)^{b}\,.
\eean
This is the constant term for Morris's identity with 
\bea
\al =a+b+c+1-p, \quad \bt = p-1-a-c, \quad \ga=c.
\eea
By assumptions, these integers are non-negative, and Morris' identity can be applied to evaluate
\eqref{SnM}. The identity gives
\bean
\label{SnMM}
&&
S(a,b,c;1,1) \,=\, (-1)^{c}\,\frac{(2c)!}{c!}\,
 \\
 \notag
&&
\phantom{aaa}
\times\,
\frac{b!\,(b +c)!}
{(p-1-a)!\, (p-1-a-c)!\,(a+b  +c -p+1)!\,(a+b+2c-p+1)!}\,.
\eean
This is an element of  $\Z$.

\vsk.2>

If $b+c\leq p-1$, then
$a+b+2c +1-p\leq p-1$. In this case, all factorials in \eqref{SnMM} except $(2c)!$
are factorials of non-negative integers which are less than $p$.  We have the following identity in $\F_p$\,:
\bea
&&
(-1)^{c}\,\frac{(2c)!}{c!}\,
\frac{b!\,(b +c)!}
{(p-1-a)!\, (p-1-a-c)!\,(a+b  +c -p+1)!\,(a+b+2c-p+1)!}\,
\\
&&
\quad
=
\frac{(2c)!}{c!}\,\frac{a!\,(a+c)!\,b!\,(b+c)!}{(a+b+c-p+1)!\,(a+b+2c-p+1)!}\,,
\eea
which is obtained by using the identities $a!(p-1-a)!=(-1)^{a+1}$ and
$(a+c)!(p-1-a-c)!=(-1)^{a + c+1}$. This proves part (i).

\vsk.2>
Proof of part (ii).
If $b+c \geq p$ and $a+b+2c\geq 2p-1$, then $(b+c)!$ has exactly one factor $p$ and
$(a+b+2c-p+1)!$ has exactly  one factor $p$. Canceling these factors and using Wilson's theorem,
we obtain the following identity in $\F_p$,
\bea
&&
(-1)^{c}\,\frac{(2c)!}{c!}\,
\frac{b!\,(b +c)!}
{(p-1-a)!\, (p-1-a-c)!\,(a+b  +c -p+1)!\,(a+b+2c-p+1)!}\,
\\
&&
\quad
=
\frac{(2c)!}{c!}\,\frac{a!\,(a+c)!\,b!\,(b+c-p)!}{(a+b+c-p+1)!\,(a+b+2c-2p+1)!}\,.
\eea
This proves part (ii).
\vsk.2>

Proof of part (iii).
If $b+c \geq p$ and $a+b+2c\leq  2p-2$, then $(b+c)!$ has exactly one factor $p$ while the other factorials in 
\eqref{SnMM} except $(2c)!$ are not divisible by $p$. This implies that $\bar S(a,b,c;1,1)=0$. The theorem is proved.
\end{proof}

\subsection{$p$-cycle $[2,2]_p$}

In this section we evaluate $\bar S(a,b,c;2,2)$.

\begin{lem}
\label{lem i22}
We have $ S(a,b,c;2,2)=0$, if at least one of the following three inequalities holds:
\bean
\label{11 ine}
a+b+c \leq 2p-2, \qquad a+c\leq p-1, \qquad
b+c\leq p-1.
\eean

\end{lem}

\begin{proof}

If $a+b+c\leq 2p-2$, then the monomial $x_1^{2p-1}x_2^{2p-1}$ does not enter the master polynomial and hence  $ S(a,b,c;2,2)=0$.

If $a+c\leq p-1$ or $b+c\leq p-1$, then $a+b+c\leq  2p-2$, and hence  $ S(a,b,c;2,2)=0$.
\end{proof}

\begin{thm}
\label{thm 2d2}

Assume that $a,b,c$ satisfy \eqref{less p} and the  inequality
\bean
\label{22 inee}
 a+b+c \geq 2p-1.
\eean
 Then the following statements hold true.
\begin{enumerate}
\item[(i)] 

If $a+b+2c \leq 3p-2$, 
then 
\bean
\label{22-}
\bar S(a,b,c;2, 2) = - \frac{(2c)!}{c!}\,\frac{a!\,(a+c-p)!\,b!\,(b+c-p)!}{(a+b+c-2p+1)!\,(a+b+2c-2p+1)!}\,.
\eean
This expression is non-zero if and only if   $2c<p$.

\item[(ii)]  
If $3p-1\leq a+b+2c$, 
then $2c>p$ and
\bean
\label{22+}
S(a,b,c; 2,2) = - \frac{(2c-p)!}{c!}\,\frac{a!\,(a+c-p)!\,b!\,(b+c-p)!}{(a+b+c-2p+1)!\,(a+b+2c-3p+1)!}\,.
\eean
This expression is non-zero.

\end{enumerate}

\end{thm}

\begin{proof}

If $a+b+c \geq 2p-1$, then $a+c \geq p$ and $b+c\geq p$.

We have 
\bean
\label{SnM22}
S(a,b,c;2,2) \,=\,
\on{CT}\,
(x_1-x_2)^{2c}
\prod_{i=1}^2 x_i^{a+1-2p}(1-x_i)^{b}\,.
\eean
This is the constant term for Morris's identity with 
\bea
\al =a+b+c+1-2p, \quad \bt = 2p-1-a-c, \quad \ga=c.
\eea
By assumptions, these integers are non-negative, and Morris' identity can be applied to evaluate
\eqref{SnM22}. The identity gives
\bean
\label{SnMM22}
&&
S(a,b,c;2,2) \,=\, (-1)^{c}\,\frac{(2c)!}{c!}\,
 \\
 \notag
&&
\phantom{aaa}
\times\,
\frac{b!\,(b +c)!}
{(2p-1-a)!\, (2p-1-a-c)!\,(a+b  +c -2p+1)!\,(a+b+2c-2p+1)!}\,.
\eean
This is an element of  $\Z$.

\vsk.2>

Proof of part (i). We have in $\F_p$ that
\bea
\frac{(b+c)!}{(2p-1-a)!\, (2p-1-a-c)!} = (-1)^{c+1} (b+c-p)!a!(a+c-p)!\,.
\eea
Since $a+b+2c \leq 3p-2$, the factorials  $(a+b  +c -2p+1)!$ and $(a+b+2c-2p+1)!$ are factorials of
 non-negative integers which are less than $p$. Then
 \bea
 \bar S(a,b,c;2, 2) = - \frac{(2c)!}{c!}\,\frac{a!\,(a+c-p)!\,b!\,(b+c-p)!}{(a+b+c-2p+1)!\,(a+b+2c-2p+1)!}\,.
\eea
Part (i) is proved.

Proof of part (ii).  We have $2c>p$ since  $3p-1\leq a+b+2c$.  In this case we have in $\F_p$ that
\bea
\frac{(2c)!}{(a+b+2c-2p+1)!} = \frac{(2c-p)!}{(a+b+2c-3p+1)!} \,,
\eea
and hence $ \bar S(a,b,c;2, 2) $ is given by formula \eqref{22+}.
\end{proof}

\subsection{$p$-cycle $[1,2]_p$}

Denote 
\bean
\label{dl}
\dl = a+b+2c+1 - 2p.
\eean

\begin{lem}
\label{lem 12}

Assume that $0<a,b, c<p$.
\begin{enumerate}
\item[(i)]  If $\dl<0$, then $S(a,b,c;1,2)=0$.

\item[(ii)]  If $\dl=0$ and $a+b<p-1$, then
$S(a,b,c;1,2)=0$.

\item[(iii)] Let $\dl=0$ and $a+b\geq p-1$, then
\bean
\label{S12}
&&
\\
\notag
&&
\phantom{aaaaa}
\bar S(a,b,c;1,2) = 
 - \frac{(2c-1)!}{(c-1)!}\,
\frac{a!\,(a+c)!\,b!\,(b+c-p)!}{(a+b+c-p+1)!\,(a+b+2c -2p+1)!}\,,
\quad \on{if}\
b+c\geq p,
\\
\label{S12^}
&&
\\
\notag
&&
\phantom{aaaaa}
\bar S(a,b,c;1,2) = - \frac{(2c-1)!}{(c-1)!}\,
\frac{a!\,(a+c-p)!\,b!\,(b+c)!}{(a+b+c-p+1)!\,(a+b+2c -2p+1)!}\,,
\quad \on{if}\
a+c\geq p.
\eean
Moreover, in formulas \eqref{S12} and \eqref{S12^} we have 
\bean
\label{S12**}
\bar S(a,b,c;1,2) = 
(-1)^{b+1} \frac{a!\,b!}{(a+b-p+1)!}\,.
\eean

\end{enumerate}

\end{lem}

\begin{cor}
If $a+c<p$ and $b+c\leq p-1$, then $\delta<0$ and hence $\bar S(a,b,c;1,2)=0$.

\end{cor}

\begin{proof}

If  $a+b+2c < 2p-1$, then $x_2^{2p-1}$ cannot be reached and $S(a,b,c;1,2)=0$.
This proves part (i).  If $a+b+2c=2p-1$, then
\bea
\int_{[1,2]_p} (x_1-x_2)^{2c}\prod_{i=1}^2x_i^a(1-x_i)^bdx_1dx_2
=(-1)^b\int_{[1]_p}x_1^a(1-x_1)^bdx_1.
\eea
The second $\F_p$-integral is zero, if $a+b<p-1$ (that implies part (ii)) and equals
$- \frac{a!b!}{(a+b-p+1)!}$ otherwise. 
Thus if  $a+b\geq p-1$, then  $S(a,b,c;1,2)$ is given by formula \eqref{S12**}.

\vsk.2>
We have
\bea
&&
\phantom{aaaaaaaaaaaa}
\frac1{(a+b-p+1)!} = \frac1{(p-1-(2c-1))!}= (2c-1)!\,,
\\
&&
(a+b+c-p+1)! =(p-1-(c-1))!\,,
\qquad 
(a+b+2c-2p+1)! = (0)!=1, 
\\
&&
\phantom{aaaaaaaaaaaaaaaaaaaa}
(b+c-p)! = (p-1-(a+c))!\,.
\eea
Hence
 $(a+c)! (b+c-p)! = (-1)^{a+c+1}$. 
Applying these identities to formula \eqref{S12**} we obtain   \eqref{S12}.
Formula \eqref{S12^} is proved similarly.
\end{proof}

\begin{thm}
\label{thm 2d12}

Assume that $0<a,b, c<p$ and $0<\dl$.
\begin{enumerate}

\item[(i)] 

If $\quad
2c<p,   \quad
a+c \leq p-1,  \quad b+c \geq p$, \quad
then 
\bean
\label{S12*}
\bar S(a,b,c;1,2) 
= -\, \frac{(2c-1)!}{(c-1)!}\,
\frac{a!\,(a+c)!\,b!\,(b+c-p)!}{(a+b+c-p+1)!\,(a+b+2c-2p+1)!}\,.
\eean

\item[(ii)] If $\quad
2c<p,   \quad
a+c \geq p,  \quad b+c < p$, \quad
then 
\bean
\label{S121}
\bar S(a,b,c;1,2) 
= -\, \frac{(2c-1)!}{(c-1)!}\,
\frac{a!\,(a+c-p)!\,b!\,(b+c)!}{(a+b+c-p+1)!\,(a+b+2c-2p+1)!}\,.
\eean

\item[(iii)]

If \ $2c<p,  \quad a+b+c < 2p-1, \quad
a+c \geq p,  \quad b+c \geq  p$, \
then  $\bar S(a,b,c;1,2)=0$.

\item[(iv)]
 If \ $2c<p$, \ \  $a+b+c \geq  2p-1$.
 Then \quad $a+c\geq p$, \quad $b+c \geq p$ \quad and
\bean
\label{S12.1}
\phantom{aaaaaa}
\bar S(a,b,c;1,2)
&=&
\frac{(2c-1)!}{(c-1)!}\,
\frac{a!\,(a+c-p)!\,b!\,(b+c-p)!}{(a+b+c-2p+1)!\,(a+b+2c-2p+1)!}\,.
\eean

\item[(v)]
If \  $2c>p$, \  $a+c \geq p$, \ then \ $\bar S(a,b,c;1,2)=0$.

\item[(vi)]
If\  \ $2c>p$,\  \ $b+c \geq p$,\   \ then \ $\bar S(a,b,c;1,2)=0$.

\end{enumerate}

\end{thm}

\begin{proof}

Under the assumptions of part (i), we apply the  recurrence relations \eqref{Ar1} 
and obtain
\bean
\label{3.6.1}
\phantom{aaa}
\bar S(a,b,c;1,2)
& =& 
\bar S(a-\delta, b,c;1,2)\,
\\
\notag
&\times &
\frac{a(a-1)\dots(a-\delta+1)\,(a+c)(a+c-1)\dots(a+c-\delta+1)}
{(a+b+c+1)(a+b+c)\dots(a+b+c-\delta+2)\,\delta!}\,.
\eean
Notice that $p>a-\dl=(p-1-b)+ (p-2c)>0$. 
We  check that all the factors in this formula are not divisible by $p$.

Indeed,  the product $(a+b+c+1) (a+b+c)\dots
(a+b+c-\delta+2)$ is non-zero in $\F_p$, since
 the first factor $a+b+c+1$ is less than $2p$ and the last factor
$a+b+c-\delta+2 = p+(p-c+1)$ is greater than $p$.
We also have $a+c-\dl+1=2p-b-c>0$.

We evaluate $S(a-\delta, b,c;1,2)$ in \eqref{3.6.1} by formula \eqref{S12} and obtain
\eqref{S12*}. Part (i) is proved.

\vsk.2>

The proof of part (ii) uses the relation \eqref{Ar2} and is similar to the proof of part (i).

\vsk.2>

To prove part (iii) denote $a+c=p+\al,\ b+c=p+\beta$. Then $\delta = \al+\beta+1$.
The product $(a+b+c+1) (a+b+c)\dots
(a+b+c-\delta+2)$ is non-zero in $\F_p$ as before.
We also have  $a+b - \delta= a+b - (a+b+2c-2p+1) = 2p-1-2c>0$.
Now we apply the recurrence  relations and write
\bea
S(a,b,c;1,2) = S(a-1,b,c;1,2)\,\frac{a(p+\al)}{(a+b+c+1)\delta}\,,
\eea
if $a>0$, or write 
\bea
S(a,b,c;1,2) = S(a,b-1,c;1,2)\,\frac{b(p+\beta)}{(a+b+c+1)\delta}\,.
\eea
 Then we apply the same transformations to that
$S(a-1,b,c;1,2)$ or $S(a,b-1,c;1,2)$ which was obtained after the first transformation. 
Repeat this procedure $\delta$ times.
As a result we will obtain a formula
$S(a,b,c;1,2) = C\, S(a',b',c;1,2)$, where $a'+b'+2c=2p-1$ and
$C$ is a ratio, whose denominator equals 
$(a+b+c+1)(a+b+c)\dots(a+b+c-\delta+2)\,\delta!$ and the numerator equals zero.
Part (iii) is proved.

\vsk.2>
Proof of part (iv).  If $a=p-1$, then
\bean
\label{p12}
\bar S(p-1,b,c; 1,2) 
&=& \int_{[2]_p} x_2^{2c+p-1} (1-x_2)^bdx_2
\\
\notag
&=&  \int_{[1]_p} x_2^{2c-1} (1-x_2)^bdx_2 = -\frac{(2c-1)!b!}{(b+2c-p)!}\,.
\eean
That formula agrees with \eqref{S12.1}.

If $a<p-1$,  we use the  recurrence  relations and write
\bea
S(a,b,c;1,2)= S(p-1, b,c;1,2) \,\prod_{i=1}^{p-1-a}\frac{(a+b+c +1+i) (a+b+2c +1+i)}{(a+i)(a+c+i)}\,.
\eea
Notice that all factors in the last product are non-zero in $\F_p$. Indeed for the smallest factor in the numerator we have
$a+b+c+1+1\geq 2p+1$ and for the largest we have
$a+b+2c+1+p-1-a =p+b+2c<3p$. In the denominator we have $p< a+c+1$ and $a+c + p-1-a =p-1+c<2p$.

 We have the following identities in $\F_p$:
 \bea
&&
\prod_{i=1}^{p-1-a} (a+b+c+1+i)=\frac{(b+c-p)!}{(a+b+c -2p+1)!}\,,
\\
&&  \prod_{i=1}^{p-1-a} (a+b+2c+1+i)=\frac{(b+2c-p)!}{(a+b+2c -2p+1)!}\,,
 \\
 &&
  \prod_{i=1}^{p-1-a} (a+i)=\frac{-1}{(a)!}\,,
 \qquad
 \prod_{i=1}^{p-1-a} (a+c+i)=\frac{(c-1)!}{(a+c-p)!}\,.
\eea
  Together with \eqref{p12} they prove part (iv).
\vsk.2>

\vsk.2>

Proof of  part (v).  If $a=p-1$. Then 
$S(p-1,b, c;1,2) = (-1)^b\int_{[2]_p} x_2^{p-1 +2c}(1-x_2)^b =0$ since
$p-1+2c>2p-1$.

Let $a<p-1$. Then
\bea
S(a,b,c;1,2) = S(p-1, b,c;1,2)\,\prod_{i=1}^{p-1-a}\frac{(a+b+c +1+i) (a+b+2c +1+i)}{(a+i)(a+c+i)}\,.
\eea
Notice that $\prod_{i=1}^{p-1-a}(a+i) = (a+1) \dots (p-1) \ne 0$ in $\F_p$. 
We have $a+c+1>p$ and
$a+c+p-1-a= p-1+c < 2p$. Hence
$\prod_{i=1}^{p-1-a}(a+c+i)  \ne 0$ in $\F_p$.
We also have
$S(p-1, b,c;1,2) = (-1)^b\int_{[2]_p} x_2^{p-1 +2c}(1-x_2)^b =0$ since
$p-1+2c>2p-1$.   Hence $\bar S(a,b,c;1,2)=0$. Part (v) is proved.

\vsk.2>
Part (vi) is proved similarly to part (v).
\end{proof}

\subsection{$p$-cycle $[1,3]_p$}

\begin{thm}
\label{thm 2d13}

Assume that $0<a,b, c<p$.
\begin{enumerate}
\item[(i)]  If $a+b+2c < 3p-1$, then $S(a,b,c;1,3)=0$. In particular if $2c<p$,  then $S(a,b,c;1,3)=0$.

\item[(ii)]  If $a+b+2c \geq 3p-1$, then
\bean
\label{S13}
\phantom{aaaa}
\bar S(a,b,c;1,3) = \,
  \frac{(2c-1-p)!}{(c-1)!}\,
\frac{a!\,(a+c-p)!\,b!\,(b+c-p)!}{(a+b+c-2p+1)!\,(a+b+2c -3p+1)!}\,.
\eean

\end{enumerate}

\end{thm}

\begin{proof}
Part (i) is clear. 

Proof of  part (ii).  Let $\delta= a+b+2c-3p+1$, then
$a-\delta +1 = 3p-b-2c>0$, $a+c-p-\delta+1 =2p-b-c>0$.
 Hence
\bean
\label{SS}
&&
\\
\notag
&&
\bar S(a,b,c;1,3) 
= \bar S(a-\delta,b,c;1,3) \,\prod_{i=1}^{\delta}\frac{(a+1-i)\,(a+c+1-i-p)}
{(a+b+c+2-i-2p)\,(a+b+2c +2-i-3p) }\,,
\\
\notag
&&
\bar S(a-\delta,b,c;1,3) 
= 
(-1)^b\,\int_{[1]_p} x_1^{3p-1-b-2c}(1-x_1)^b dx_1= (-1)^{b+1}\frac{(3p-1-b-2c)!\,b!}
{(2p-2c)!}\,.
\eean
Notice that $3p-1-b-2c \leq p-1$. Indeed, if  $3p-1-b-2c > p-1$, then $a+b+2c < 3p-1$, that contradicts to the assumptions.

We also have the following identities in $\F_p$\,;
\bea
&&
(3p-1-b-2c)!\prod_{i=1}^{\delta}(a+1-i) = (3p-1-b-2c)!\, a(a-1)\dots (3p-b-2c)= a!\,,
\\
&&
\prod_{i=1}^{\delta}(a+c+1-i-p) =  (a+c-p)(a+c-p-1)\dots (2p-b-c) 
\\
&&
= (a+c-p)!\,(b+c-p)!(-1)^{b+c},
\\
&&
\prod_{i=1}^\delta(a+b+c +2-i-2p) = (a+b+2c +1-i-2p)\dots(p+1-c)
\\
&&
= (a+b+c+1-2p)!\,(c-1)!(-1)^c.
\qquad
\frac 1{(2p-2c)!}=- (2c-1-p)!\,.
\eea
 These formulas imply part (ii).
\end{proof}

\subsection{$p$-cycle $[2,3]_p$}

\begin{thm}
\label{thm 2d23}

Assume that $0 < a,b,c<p$. Then  $\bar S(a,b,c;2,3)=0. $

\end{thm}

\begin{proof}
Clearly $S(a,b,c;2,3) =0$ if $a+b+2c<3p-1$. If $\delta =a+b+2c-3p+1\geq 0$,
then
\bea
S(a,b,c;2,3) 
= S(a-\delta,b,c; 2,3) \,\prod_{i=1}^{\delta}\frac{(a+1-i)\,(a+c+1-i-p)}
{(a+b+c+2-i - 2p)\,(a+b+2c +2-i-3p) }\,,
\\
S(a-\delta,b,c;2,3) 
= 
(-1)^b\,\int_{[2]_p} x_1^{3p-1-b-2c}(1-x_1)^bdx_1,
\phantom{aaaaaaaaaaaaaaaaaaaa}
\eea
where the last $\F_p$-integral is zero since $3p-1-b-2c \leq p-1$.
\end{proof}

\subsection{$p$-cycles $[l_1,l_2]_p$}

\begin{thm}
\label{thm 2dl1l2}

Assume that $0< a,b,c<p$ and  $l_1\leq l_2$. Then  $\bar S(a,b,c;l_1,l_2)=0$
if  $(l_1,l_2)\notin \{(1,1), (2,2), (1,2), (1,3)\}$.

\end{thm}

\begin{proof} 
We have $\bar S(a,b,c;3,3)=0$, since the total degree of
$\Phi(x_1,x_2; a,b,c)$ is $2(a+b+c)$ and it is less than the total degree
$6p-2$ of the monomial $x_1^{3p-1}x_2^{3p-1}$. 

Let $l_1\leq l_2$. If $l_2\geq 4$, then
$a+b+2c< l_2p -1$, and $\bar S(a,b,c;l_1,l_2)=0$. All other cases are discussed in the previous theorems.
\end{proof}

\subsection{Relations between $p$-cycles}

The next theorem lists all the integers $0<a,b,c<p$ such  that there are more than one pair $l_1\leq l_2$ with non-zero
$\bar S(a,b,c;l_1,l_2)$. 

\begin{thm}
\label{thm rels}

If  $0<a,b,c<p$ and there are more than one pair $l_1\leq l_2$ such that
$\bar S(a,b,c;l_1,l_2)$ is non-zero, then all such $(a,b,c;l_1,l_2)$ are listed below.
\begin{enumerate}
\item[(i)]

If 
\bean
\label{R1}
\phantom{aaa}
2c<p, \ a+c\leq p-1, \ b+c\geq p,  \  a+b+2c\geq 2p-1,
\eean
then
$\bar S(a,b,c;1,1)$,  $\bar S(a,b,c;1,2)$, $\bar S(a,b,c;2,1)$ are non-zero and
\bean
\label{rel 11-12}
&&
\phantom{aaaa}
-\frac 12\,\bar S(a,b,c;1,1) = \bar S(a,b,c;1,2) = \bar S(a,b,c;2,1),
\\
\label{exc1}
&&
\bar S(a,b,c;1,1) = \frac{(2c)!}{c!}\,\frac{a!\,(a+c)!\,b!\,(b+c-p)!}{(a+b+c-p+1)!\,(a+b+2c-2p+1)!}\,.
\eean

\item[(ii)]

If 
\bean
\label{R2}
2c<p, 
\quad a+b+c \geq  2p-1, 
\eean
 then
$\bar S(a,b,c;2,2)$,  $\bar S(a,b,c;1,2)$, $\bar S(a,b,c;2,1)$ are non-zero and
\bean
\label{rel 22-12}
&&
\phantom{aaaa}
-\frac 12\,\bar S(a,b,c;2,2) = \bar S(a,b,c;1,2) = \bar S(a,b,c;2,1).
\\
\label{exc2}
&&
\bar S(a,b,c;2,2) =
 - \frac{(2c)!}{c!}\,\frac{a!\,(a+c-p)!\,b!\,(b+c-p)!}{(a+b+c-2p+1)!\,(a+b+2c-2p+1)!}\,.
\eean

\item[(iii)]

If 
\bean
\label{R3}
2c>p, \quad   a+b+2c\geq  3p-1,
\eean 
then
$\bar S(a,b,c;2,2)$, $\bar S(a,b,c;1,3)$, $\bar S(a,b,c;3,1)$
are non-zero and
\bean
\label{rel 22-13}
&&
\phantom{aaaa}
-\frac 12\,\bar S(a,b,c;2,2) =\bar S(a,b,c;1,3) = \bar S(a,b,c;3,1),
\\
\label{exc3}
&&
\bar S(a,b,c;2,2) =
 - \frac{(2c-p)!}{c!}\,\frac{a!\,(a+c-p)!\,b!\,(b+c-p)!}{(a+b+c-2p+1)!\,(a+b+2c-3p+1)!}\,.
\eean

\end{enumerate}

If $(a,b,c)$  does not satisfy the system of inequalities
$2c<p, \ a+c\leq p-1, \ b+c\geq p,  \  a+b+2c\geq 2p-1$,
and does not satisfy the  system of  inequalities
$2c<p, \ a+b+c \geq  2p-1$,  
and does not satisfy the  system of  inequalities
$2c>p, \   a+b+2c\geq  3p-1$,
then there exists at most one pair
$(l_1,l_2)$ with $l_1\leq l_2$ such that $\bar S(a,b,c;l_1,l_2)\ne 0$.

\end{thm}

\begin{proof}  
The proof is by inspection of Theorems \ref{thm 2d1}, \ref{thm 2d2}, \ref{thm 2d12}, \ref{thm 2d13}.
In particular,  part (i) follows from part (ii) of Theorem \ref{thm 2d1} and part (i) of Theorem \ref{thm 2d12}.
Part (ii) follows from part (i) of Theorem \ref{thm 2d2} and part (iv) of Theorem \ref{thm 2d12}.
Part (iii) follows from part (ii) of Theorem \ref{thm 2d2} and part (ii) of Theorem \ref{thm 2d13}.
\end{proof}

\begin{ex} 
 For $p=7$, we have $\bar S(3,4,3;1,1)=1$,  $\bar S(6.6,3;2,2)=2$, 
 $\bar S(6,6,6;2,2)=5$ by formulas \eqref{exc1}, \eqref{exc2}, \eqref{exc3}, respectively.

\end{ex}

\begin{rem}
In analogy with   complex case one may think that the master polynomial defines a local system 
(depending on parameters $a,b,c$)
on the the two-dimensional affine space $\F_p^2$ with coordinates $x_1,x_2$ and with 
one-dimensional $p$-cohomology space. 
If $2c<p, \ a+c\leq p-1, \ b+c\geq p,  \  a+b+2c\geq 2p-1$, then  the $p$-cycles
 $[1,1]_p$, $[1,2]_p$, $[2,1]_p$ give non-zero elements of the 
 one-dimensional dual $p$-homology space and
  $2[1,1]_p + [1,2]_p + [2,1]_p \thicksim 0$, cf. \eqref{rel 11-12}.  
  It would be interesting to find that $p$-chain, whose $p$-boundary is $[1,1]_p + [1,2]_p + [2,1]_p$\,.

  The same question may be addressed to the relations in 
\eqref{rel 22-12} and \eqref{rel 22-13}.  The relation in \eqref{rel 22-13} is discussed below.

\end{rem}

\subsection{Relation \eqref{rel 22-13}} 
\label{sec 4.8}

In this section we prove \eqref{rel 22-13} without explicitly evaluating its terms.
Under assumptions \eqref{R3}, consider the two polynomials in $\F_p[x_1,x_2]$:
\bean
\label{exp1}
(x_1-x_2)^{2c} \prod_{i=1}^2x_i^a (1-x_i)^b  
&=&
 \sum_{d}\, \al_{d_1,d_2} \,x_1^{d_1} x_2^{d_2}.
\\
(x_1-x_2)^{2c-p} \prod_{i=1}^2x_i^a (1-x_i)^b  
&=&
 \sum_{d}\, \bt_{d_1,d_2} \,x_1^{d_1} x_2a^{d_2}.
\eean
By  the previous theorems, the first polynomial has exactly three non-zero coefficients
$\al_{d_1,d_2} $ with $(d_1,d_2)$ of the form $(l_1p-1, l_2p-1)$ for some positive integers $l_1,l_2$. These coefficients are
\bea
\al_{3p-1,p-1} =\bar S(a,b,c;3,1), \quad
\al_{2p-1,2p-1} =\bar S(a,b,c;2,2), \quad
\al_{p-1,3p-1} =\bar S(a,b,c;1,3).
\eea
Hence, the second polynomial has exactly two non-zero coefficients
$\bt_{d_1,d_2} $ with $(d_1,d_2)$ of the form $(l_1p-1, l_2p-1)$ for some positive integers $l_1,l_2$. These coefficients are
$\bt_{p-1,2p-1}$ and  $\bt_{2p-1,p-1}$\,, \  and
\bea
\al_{3p-1,p-1} =\bt_{2p-1, p-1}, \quad
\al_{2p-1,2p-1} =\bt_{p-1, 2p-1} - \bt_{2p-1, p-1},
\quad
\al_{p-1,3p-1} = - \bt_{p-1, 2p-1}.
\eea
We also have   $\bt_{p-1, 2p-1} = - \bt_{2p-1, p-1}$ since
the second polynomial is skew-symmetric in $x_1,x_2$. 
Hence,
$\al_{3p-1,p-1} =\bt_{2p-1, p-1}$, 
$\al_{2p-1,2p-1} = - 2\bt_{2p-1, p-1}$,
$\al_{p-1,3p-1} = \bt_{2p-1, 1p-1}$ and
\bean
\label{RR3}
-\frac 12 \al_{2p-1,2p-1} = \al_{p-1,3p-1} = \al_{3p-1,p-1}\,.
\eean

\end{document}